\documentclass[12pt]{article}
\usepackage{xypic}
\usepackage{epsfig}
\usepackage{tikz}
\usepackage{graphicx}
\usepackage{amsthm}
\usepackage{amssymb}
\usepackage{caption}
\usepackage{subcaption}
\usepackage{color}
\usepackage{hyperref}
\usepackage{amsmath,mathrsfs,amsfonts,verbatim,enumitem,leftidx}

\newtheorem{cro}{Corollary}[section]
\newtheorem{defn}{Definition}[section]
\newtheorem{prop}{Proposition}[section]
\newtheorem{thm}{Theorem}[section]
\newtheorem{lem}{Lemma}[section]
\newtheorem{rem}{\bf Remark}[section]

\numberwithin{equation}{section}

\begin{document}

\title{Entropy and Emergence of Topological Dynamical Systems 
 \footnotetext {* Corresponding author}
  }
\author{Yong Ji, Ercai Chen and Xiaoyao Zhou\\
	\small   School of Mathematical Sciences and Institute of Mathematics, Nanjing Normal University,\\
	\small   Nanjing 210046, Jiangsu, P.R.China\\
	\small    e-mail: imjiyong@126.com
	ecchen@njnu.edu.cn
	zhouxiaoyaodeyouxian@126.com
}
\date{}
\maketitle

\begin{center}
 \begin{minipage}{120mm}
{\small {\bf Abstract.}
	A topological dynamical system $(X,f)$ induces two natural systems, one is on the probability measure spaces and other one is on the hyperspace.
	We introduce a concept for these two spaces, which is called entropy order, and prove that it coincides with topological entropy of $(X,f)$. We also consider the entropy order of an invariant measure and a variational principle is established.
   }
\end{minipage}
 \end{center}

\vskip0.5cm

\section{Introduction}
Throughout this paper, by a topological dynamical system we mean a pair $(X,f)$, where $X$ is a compact metric space and $f$ is a continuous map from
$X$ to itself. We use $\mathcal{M}(X)$, $\mathcal{M}_f(X)$, and $\mathcal{M}_f^{erg}(X)$ to denote the set of Borel probability measures, the set of $f$-invariant Borel probability measures, and the set of ergodic $f$-invariant measures with the weak*-topology, respectively. Let $\mathcal{K}(X)$ denote the space of nonempty closed subsets of $X$ provided with the Hausdorff metric.

A topological dynamical system $(X,f)$ induces two natural systems: $(\mathcal{M}(X), f_{\mathcal{M}})$ and $(\mathcal{K}(X),f_{\mathcal{K}})$.
The study of the connections between the dynamical properties of $(X,f)$ and the induced topological dynamical systems has attracted a lot of interest. In 1975,
Bauer and Sigmund \cite{Bauer1975} first gave a systematic study of the dynamical properties of the induced dynamical systems. It was shown that $(X,f)$ is weakly mixing (mildly mixing, strongly mixing) if and only if $(\mathcal{M}(X), f_{\mathcal{M}})$ (or $(\mathcal{K}(X),f_{\mathcal{K}})$) has the same property. In 2005, Banks \cite{Banks2005} proved that $(\mathcal{K}(X),f_{\mathcal{K}})$ is transitive if and only if $(X,f)$ is weakly mixing.
Recently, Li,  Yan and Ye \cite{LiYan2015}   investigated the dynamical properties of the connections between $(X,f)$, $(\mathcal{M}(X), f_{\mathcal{M}})$ and $(\mathcal{K}(X),f_{\mathcal{K}})$, focussing on periodic systems, P-systems, M-systems, E-systems, and   disjointness. Later, Li,  Oprocha, Ye and  Zhang  \cite{LiOprocha2017} showed that $(\mathcal{K}(X),f_{\mathcal{K}})$ is pointwise minimal if and only if $(X,f)$  is equicontinuous
and that $(\mathcal{K}(X),f_{\mathcal{K}})$ is weakly rigid if and only if $(X,f)$ is uniformly rigid.  The above investigations focussed on the qualitative study of the relations between the complexity of original system and that of its induced systems.

Inspired by the concept of Shannon entropy in information theory, Kolmogorov and Sinai introduced measure-theoretic entropy into ergodic theory.
Later, Adler, Konheim and McAndrew \cite{Adler1965} proposed a notion of topological entropy whose definition does not involve any invariant measure. After that,
the notion of entropy has played a crucial role in quantifying the degree of disorder in a system.
For a topological dynamical system $(X,f)$, by embedding the product of system into induced systems,  
Bauer and Sigmund \cite{Bauer1975} showed
\begin{equation}\label{>0}
h_{top}(X,f)>0\Rightarrow h_{top}(f_{\mathcal{M}})=h_{top}(f_{\mathcal{K}})=\infty.
\end{equation} 
In 1995,  Glasner and  Weiss \cite{Glasner1995} obtained a remarkable result:
\begin{equation}\label{glasner}
h_{top}(X,f)=0\Rightarrow h_{top}(f_{\mathcal{M}})=0.
\end{equation} 
Later, this connection was further studied by Kerr and Li \cite{Kerr2005} by shown $(X,f)$ is null if and only if $(\mathcal{M}(X),f_{\mathcal{M}})$ is null.
In 2017, Qiao and Zhou \cite{Qiao2017} extended this result to  the case of sequence entropy and showed the upper entropy dimension of the $(X,f)$ is equal to that of $(\mathcal{M}(X),f_{\mathcal{M}})$. Recently, Liu, Qiao and Xu \cite{Liu2020} proved that if the topological entropy of a nonautonomous dynamical system $(X,\{f_n\}_{n=1}^\infty)$ vanish, then so does that of its induced system $(\mathcal{M}(X),\{f_n\}_{n=1}^\infty)$; and once the topological entropy of $(X,\{f_n\}_{n=1}^\infty)$ is positive, that of its induced system $(\mathcal{M}(X),\{f_n\}_{n=1}^\infty)$ jumps to infinity.

While the topological property of induced transformations has been extensively studied, their quantitative relationship with underlying systems is rarely considered. From formulas (\ref{>0}) and (\ref{glasner}), the induced systems are more complicated. 

Emergence is one of the most important concept in complexity science \cite{MacKay2008}.
It was introduced by Berger \cite{Berger2017} in a compatible viewpoint to quantify the complexity of the statistical behavior of typical orbits for general dynamical systems. And further developed in \cite{Berger2019} with Bochi. Let $(X,f)$ be a topological dynamical system.
\begin{defn}\rm(Metric emergence \cite{Berger2017})
	For any $\mu\in\mathcal{M}(X)$, the metric emergence $\mathcal{E}_\mu(f)$ of $f$ at scale $\varepsilon>0$ is the minimal number $N$ of probability measure $\{ \nu_i \}_{1\leq i\leq N}$ so that 
	$$\limsup_{n\to\infty}\int_Xd\left(e_n^f(x),\{\nu_i \}_{1\leq i\leq N} \right)\,\mathrm{d}\mu(x)<\varepsilon.$$
\end{defn}
\begin{defn}\rm(Topological emergence \cite{Berger2019})
	The topological emergence $\mathcal{E}_{top}(f)(\varepsilon)$ of $f$ is the function which associates to $\varepsilon>0$ the minimal number of $\varepsilon$-balls of $\mathcal{M}(X)$ whose union covers $\mathcal{M}_f^{erg}(X)$.
\end{defn}  

In \cite{Berger2019}, Berger and Bochi showed that the order of the topological emergence of a system $f$ on a compact manifold of dimension $d$ is at most $d$. And they provided examples of systems which the upper bound is generically attained.
\begin{thm}{\rm\cite[Theorem A]{Berger2019}}\label{mxcase}
	Let $f$ be $C^{1+\alpha}$-mapping of a manifold which admits a basic hyperbolic set $K$ with box dimension $d$. Assume that $f$ is conformal expanding or that $f$ is a conservative surface diffeomorphism. Then the topological emergence $f|K$ is stretched exponential with exponent $d$:
	\begin{equation*}
	\lim_{\epsilon\to0}\frac{\log\log\mathcal{E}_{top}(f|K)(\varepsilon)}{-\log\varepsilon}=d.
	\end{equation*}
\end{thm}
Based on these examples, they pointed out that emergence and entropy are completely unrelated, in the sense that there exist systems with positive metric entropy but minimal metric emergence, and stretched exponential emergence but zero topological entropy. Following their ideas, and inspired by the relationship between dimension and entropy, we will introduce the notions of dynamical emergence and entropy order for induced spaces $\mathcal{M}(X)$ and $\mathcal{K}(X)$. See \cite{Berger2017, Berger2020, Berger2019} for more information on emergence.

Now, we state our main results:
firstly, we show that the entropy orders of $\mathcal{M}(X)$ and $\mathcal{K}(X)$ equal to the topological entropy of $(X,f)$ . 
\begin{thm}\label{mh}
	Let $(X,f)$ be a dynamical system. Then
	$$\mathcal{E}(f_{\mathcal{M}})=h_{top}(f).$$
\end{thm}
\begin{thm}\label{K}
	Let $(X,f)$ be a dynamical system. Then
	$$\mathcal{E}(f_{\mathcal{K}})=h_{top}(f).$$
\end{thm}

Secondly, we show that for dynamical systems satisfying some conditions, the entropy order of some subset is sufficient to reflect that of the whole space. Denote the collection of ergodic probability measures supported on the orbit of a periodic point by $\mathcal{M}_f^{per}(X)$, and let $\mathcal{K}_f(X):=\{B\in\mathcal{K}(X):f(B)=B\}$, i.e., the set of fixed points of $f_{\mathcal{K}}$.
\begin{thm}\label{expansive}
	Let $(X,d)$ be a compact metric space and suppose $f:X\to X$ is a positively expansive map with specification property. Then $$h_{top}(f)=\mathcal{E}(f_{\mathcal{M}},\mathcal{M}_f^{per}(X)).$$
\end{thm}

\begin{thm}\label{Ex}
	Let $(X,d)$ be a compact metric space and $f:X\to X$ be a positively expansive map with specification property. Then $$h_{top}(f)=\mathcal{E}(f_{\mathcal{K}},\mathcal{K}_f(X)).$$
\end{thm}
In addition, we define the entropy order for invariant measures $\mathcal{E}_\mu$, and establish the following variational principle.
\begin{thm}\label{variation}
	Let $(X,d)$ be a compact metric space and $f:X\to X$ a continuous map. Then
	$$\mathcal{E}(f_{\mathcal{M}},\mathcal{M}_f^{erg}(X))=\max\{\mathcal{E}_\mu:\mu\in\mathcal{M}_f(X) \}.$$
\end{thm}

Finally, we consider the Theorem \ref{mxcase} in the case of hyperspace. Let $(X,f)$ be a topological dynamical system, denote by $N^f(\varepsilon)$ the minimum cardinality of $\varepsilon$-dense set of $\mathcal{K}_f(X)$.
\begin{thm}\label{conformal}
	Let $f$ be $C^{1+\alpha}$-mapping of manifold which admits a basic hyperbolic set $\Lambda$ with box dimension $b$. Assume that $f$ is conformal expanding, then 
	$$b=\lim_{\varepsilon\to0}\frac{\log\log N^{f|\Lambda}(\varepsilon)}{-\log\varepsilon}.$$
\end{thm}

Our results give some new points to comprehensive the connection of topological dynamical systems and their induced systems. This paper is organized as follows. In Section 2, we list basic notions and results, two kinds of dynamical metrics on induced spaces. In Section 3, we give the definitions of dynamical emergences and entropy orders. In Section 4, we prove our main results. In Section 5, we give a short argument on pointwise emergence introduced in {\rm\cite{Kiriki2019}}.

\section{Preliminaries}
In this section, we recall some basic definitions and give some properties of induced transformations $f_\mathcal{M}$ and $f_\mathcal{K}$. We shall denote by $\mathbb{N}$ the set of all positive integers, $\mathbb{N}_0$ the set of all non-negative integers, and $\mathbb{R}$ the set of all real numbers.

\subsection{Positively expansiveness and specification property}
Let $(X,f)$ be a topological dynamical system with a metric $d$ on $X$. Then $f$ is said to satisfy the Bowen specification property if for any $\varepsilon>0$, there exists an $n_0\in\mathbb{N}$ such that given any $x_1,\ldots,x_k\in X$, any nonnegative integers $n_1,\ldots,n_k$, $s_1,\ldots,s_k$ with $s_i\geq n_0$, there exists a point $p\in X$ such that, writing $m_j=\sum_{i=1}^{j}n_i+s_i$ for $j=1,\ldots,k$ and $m_0=0$,\\
(1) $d(f^{m_{j-1}+i}p,f^ix_j)<\varepsilon$ for $0\leq i<n_j$
and $1\leq j\leq k$;\\
(2) $f^{m_k}(p)=p$.\\
If the periodicity condition $(2)$ is omitted, we say that $f$ has the specification property.

A dynamical system $(X,f)$ is called positively expansive if there exists a $\rho>0$ such that for distinct $x,y\in X$, there exists an $n\in\mathbb{N}_0$ such that
$d(f^nx,f^ny)>\rho$. Call $\rho$ the expansive constant for $f$. 

The following important fact arises from the definition of expansiveness, see \cite{Kwietniak2016, Kwietniak2012} for explicit statement and proof.
\begin{lem}\label{sp+exp}
	If $(X,f)$ has the specification property and positively expansiveness, then $(X,f)$ has the Bowen specification property.
\end{lem}

\subsection{The space of measures}
Let $(X,d)$ be a compact metric space. It is well known that $\mathcal{M}(X)$ is compact with the weak*-topology. Denote the support set of a measure $\mu$ by $\text{supp}\mu$. We will next recall the $p$-Wasserstein metric, $1\leq p<\infty$, on $\mathcal{M}(X)$. Given $\mu,\nu\in\mathcal{M}(X)$, a transport plan from $\mu$ to $\nu$ is a probability measure $\pi$ on the product $X\times X$ such that $(p_1)_{\ast}\pi=\mu$ and $(p_2)_{\ast}\pi=\nu$, where $p_1, p_2: X\times X\to X$ are the canonical projections, and $(p_i)_\ast\pi=\pi\circ p_i^{-1}$ for $i=1,2$. Let $\Pi(\mu,\nu)$ denote the set of all transport plans from $\mu$ to $\nu$. Then for $1\leq p<\infty$, the $p$-Wasserstein distance between $\mu$ and $\nu$ is defined as: $$W_p(\mu,\nu):=\left( \inf_{\pi\in\Pi(\mu,\nu)}\int [d(x,y)]^p\,\mathrm{d}\pi(x,y) \right)^{1/p}.$$
The integral in the above formula is called the cost of the transport plan $\pi$, and the infimum is always attained. It is a metric and the induced topology is just the weak*-topology on $\mathcal{M}(X)$ (see \cite[Theorem 7.3 and Theorem 7.12]{Villani2003}).

Another metric on $\mathcal{M}(X)$ we consider is the L\'{e}vy--Prokhorov metric. For $\mu,\nu\in\mathcal{M}(X)$, $LP(\mu,\nu)$ is defined as the infimum of $\varepsilon>0$ such that for every Borel set $B\subset X$,
$$\mu(B)\leq \nu(B^{\varepsilon})+\varepsilon\ \text{and}\ \nu(B)\leq \mu(B^{\varepsilon})+\varepsilon,$$ where $B^{\varepsilon}$ denotes the $\varepsilon$-neighbourhood of $B$. The L\'{e}vy--Prokhorov metric can also be defined in terms of transport plans: it is the infimum of $\varepsilon>0$ such that there exists a $\pi\in\Pi(\mu,\nu)$ with the set $\{(x,y)\in X\times X:d(x,y)>\varepsilon \}$ having $\pi$-measure less than $\varepsilon$. The induced topology is also the weak*-topology for measures (see \cite[p. 74]{Villani2003}).

Note that both the Wasserstein metric and the L\'{e}vy--Prokhorov metric of $\mathcal{M}(X)$ depend on the metric of $X$. Moreover, the map $x\mapsto\delta_x$ makes $X$ embed isometrically into $\mathcal{M}(X)$.
The $p$-Wasserstein metric and the L\'{e}vy--Prokhorov metric have the following H\"{o}lder comparisons: for $1\leq q\leq p<\infty$,
\begin{align}\label{eqaivalent}
	&W_q\leq W_p\leq({\rm diam } ~X)^{1-\frac{q}{p}}W_q^{\frac{q}{p}}.\\
	&LP^{1+\frac{1}{p}}\leq W_p\leq (1+({\rm diam }~ X)^p)^{\frac{1}{p}}LP^{\frac{1}{p}},
\end{align}
see \cite[(7.4)]{Villani2003}, \cite[Theorem 2]{Gibbs2002}. Sometimes we will also need the following metric on $\mathcal{M}(X)$, which also induces the weak*-topology. Let $C(X,[0,1])$ denote the set of continuous maps on $X$ with values in $[0,1]$, and $\{\varphi_n \}_{n\in\mathbb{N}}\subset C(X,[0,1])$ be a countable dense subset. Then
$$D(\mu,\nu):=\sum_{n=1}^{\infty}\frac{|\int\varphi_n \,\mathrm{d}\mu-\int\varphi_n\,\mathrm{d}\nu|}{2^n}.$$ 

\subsection{Box dimension and metric order}
Let $X$ be a metric space. For $\varepsilon>0$ and an non-empty subset $A\subset X$, call a subset $E\subset X$ is $\varepsilon$-dense set of $A$ if $A\subset \cup_{x\in E}B(x,\varepsilon)$, call a subset $F\subset A$ is $\varepsilon$-separated of $A$ if for distinct $x,y\in F$ we have $d(x,y)>\varepsilon$. Denote by $N(A,\varepsilon)$, $S(A,\varepsilon)$ the smallest cardinality of $\varepsilon$-dense sets of $A$, the largest cardinality of $\varepsilon$-separated sets of $A$ respectively. It is easy to check that 
$$S(A,2\varepsilon)\leq N(A,\varepsilon)\leq S(A,\varepsilon).$$
The upper box dimension of $A$ is defined by 
$$\overline{\text{dim}}(A)=\limsup_{\varepsilon\to0}\frac{\log N(A,\varepsilon)}{-\log\varepsilon}=\limsup_{\varepsilon\to0}\frac{\log S(A,\varepsilon)}{-\log\varepsilon}.$$
The lower box dimension $\underline{\text{dim}}(A)$ is defined by taking liminf instead of limsup. The notion $\overline{\text{dim}}(A,d)$, $\underline{\text{dim}}(A,d)$ indicates dependence on metric $d$. If these two quantities coincide, call they box dimension of $A$ and denoted by $\text{dim}(A)$.

To deal with the case of box dimension being infinite, we follow the concept in \cite{Kolmogorov1976}. Let 
$$\text{mo}(A)=\lim_{\varepsilon\to0}\frac{\log\log N(A,\varepsilon)}{-\log\varepsilon}=\lim_{\varepsilon\to0}\frac{\log\log S(A,\varepsilon)}{-\log\varepsilon},$$
if the limit exists, and call it metric order of $A$. The lower and upper metric orders $\underline{\text{mo}}(A)$, $\overline{\text{mo}}(A)$ are defined by taking liminf and limsup. The following theorem relates the box dimension of underlying metric space and metric order of the space of probability measures. From this point, the induced space is more ``huge" than original space.

\begin{thm}{\rm\cite{Berger2019, Bolley2007, Kloeckner2012}}
	If $(X,d)$ is a compact space with well defined box dimension, then for any $p\geq1$, ${\rm dim}(X,d)={\rm mo}(\mathcal{M}(X),W_p).$
\end{thm}

\subsection{The induced transformation $f_{\mathcal{M}}$}
Let $(X,d)$ be a compact metric space and $f:X\to X$ be a continuous map. The induced map on $\mathcal{M}(X)$ is defined by
\begin{align*}
f_{\mathcal{M}}:\mathcal{M}(X)&\to\mathcal{M}(X)\\
\mu&\mapsto f_{\ast}\mu,
\end{align*}
where $f_{\ast}\mu(B)=\mu(f^{-1}B)$ for every Borel set $B$. For $1\leq p<\infty$ and $n\in\mathbb{N}$, let $W_{p,n}$ denote the $n$th Bowen metric on $\mathcal{M}(X)$, i.e.,
$$W_{p,n}(\mu,\nu)=\max_{0\leq i\leq n-1}\left\{W_p\left(f_{\mathcal{M}}^i(\mu),f_{\mathcal{M}}^i(\nu) \right)\right\}.$$ In addition, let $W_{p}^{n}$ denote the $p$-Wasserstein metric defined by the $n$th Bowen metric on $X$, i.e.,
\begin{equation}\label{wn}
	W_p^n(\mu,\nu)=\left(\inf_{\pi\in\Pi(\mu,\nu)}\int [d_n(x,y)]^p\,\mathrm{d}\pi(x,y) \right)^{1/p},
\end{equation}
where $d_n(x,y)=\max\left\{d(f^ix,f^iy): 0\leq i\leq n-1 \right\}$.  Similarly, let $LP_n$ and $LP^n$ denote the $n$th L\'{e}vy--Prokhorov metric and the L\'{e}vy--Prokhorov metric induced by $d_n$, respectively.

\begin{prop}\label{metric}
 Let $(X,f)$ be a topological dynamical system with the metric $d$ on $X$. For any $p\geq 1$, $n\in\mathbb{N}$, $$W_{p,n}\leq W_{p}^{n}.$$ 
\end{prop}
\begin{proof}
	We will prove that $W_p(f_{\mathcal{M}}^i\mu,f_{\mathcal{M}}^i\nu)\leq W_p^n(\mu,\nu)$ for $\mu,\nu\in\mathcal{M}(X)$ and any $i=0,1,\ldots,n-1$. Let $\varepsilon=W_{p}^{n}(\mu,\nu)$. Then there exists $\pi\in\Pi(\mu,\nu)$ such that $\displaystyle\int [d_n(x,y)]^p\,\mathrm{d}\pi(x,y)=\varepsilon^p$. As $d_n(x,y)\geq d(f^i(x), f^i(y))$ for $0\leq i<n$, we have $$\int[d(f^i(x),f^i(y))]^p\,\mathrm{d}\pi\leq\varepsilon^p.$$
	Consider the probability measure $(f^i\times f^i)_{\ast}\pi$ of $X\times X$. It is easy to check that $(f^i\times f^i)_{\ast}\pi\in\Pi(f^i_{\mathcal{M}}(\mu),f^i_{\mathcal{M}}(\nu))$. Moreover, $$\int[d(x,y)]^p\,\mathrm{d}(f^i\times f^i)_{\ast}\pi=\int[d(x,y)]^p\,\mathrm{d}\pi\circ(f^i\times f^i)^{-1}=\int[d(f^i(x),f^i(y))]^p\,\mathrm{d}\pi\leq\varepsilon^p,$$ i.e., $W_p(f_{\mathcal{M}}^i(\mu),f_{\mathcal{M}}^i(\nu))\leq\varepsilon$.

\end{proof}
\begin{prop}
	For any topological dynamical system $(X,f)$ with metric $d$ on $X$ and $n\in\mathbb{N}$, $$LP_n\leq LP^n.$$
\end{prop}
\begin{proof}
	Let $\varepsilon:=LP^n(\mu,\nu)$, and for a Borel set $E$, let $E^{\varepsilon}_n$ denote the $\varepsilon$-neighbourhood with metric $d_n$. Then
	$$
	\mu(E)\leq \nu(E^{\varepsilon}_n)+\varepsilon,\
	\nu(E)\leq\mu(E^{\varepsilon}_n)+\varepsilon.
	$$
	For any $i=0,1,\ldots,n-1$ and  $z\in(f^{-i}E)^{\varepsilon}_n$, we have $d(f^iz,E)<\varepsilon$, i.e. $z\in f^{-i}E^{\varepsilon}$, which leads to $(f^{-i}E)^{\varepsilon}_n\subset f^{-i}E^{\varepsilon}$. Therefore,
	$$f_{\mathcal{M}}^i\mu(E)
	=
	\mu(f^{-i}E)
	\leq
	\nu((f^{-i}E)^{\varepsilon}_n)+\varepsilon
	\leq
	\nu(f^{-i}E^{\varepsilon})+\varepsilon
	=
	f_{\mathcal{M}}^i\nu(E^{\varepsilon})+\varepsilon.$$
	Similarly, $f_{\mathcal{M}}^i\nu(E)
	\leq f_{\mathcal{M}}^i\mu(E^{\varepsilon})+\varepsilon$, and hence $LP(f_{\mathcal{M}}^i\mu,f_{\mathcal{M}}^i\nu)\leq\varepsilon$, which completes the proof.
\end{proof}

\subsection{The induced transformation $f_{\mathcal{K}}$}
Let $(X,d)$ be a compact metric space and $f:X\to X$ a continuous map. The Hausdorff metric $H$ on $\mathcal{K}(X)$ is defined by
$$H(B,C)=\inf\{\varepsilon>0:B\subset C^{\varepsilon}\ \text{and}\ C\subset B^{\varepsilon} \},\ \ B,C\in\mathcal{K}(X).$$
This metric turns $\mathcal{K}(X)$ into a compact metric space.
The induced map on $\mathcal{K}(X)$ is defined by
\begin{align*}
f_{\mathcal{K}}:\mathcal{K}(X)&\to\mathcal{K}(X)\\
B&\mapsto f(B).
\end{align*}
For $n\in\mathbb{N}$, denote by $H^n$ the Hausdorff metric induced by the $n$th Bowen metric of $(X,f)$, i.e.,
\begin{equation}\label{hn}
	H^n(B,C)=\inf\{\varepsilon>0:B\subset C_n^{\varepsilon}\ \text{and}\ C\subset B_n^{\varepsilon} \},
\end{equation}
where $A_n^{\varepsilon}=\{x\in X:d_n(x,A)<\varepsilon \}$  for a set $A$. Moreover, set $$H_n(B,C)=\max\{H(f_\mathcal{K}^iB,f_\mathcal{K}^iC):i=0,1,\ldots,n-1 \},$$ the $n$th Bowen metric for $f_{\mathcal{K}}$. Denote by $\text{diam}(B,d_n)$ the diameter of a set $B\subset X$ under the metric $d_n$.

\begin{prop}\label{fk}
	Let  $(X,f)$ be a topological dynamical system. For $n\in\mathbb{N}$, $B,C\in\mathcal{K}(X)$, we have $$H_n(B,C)\leq H^n(B,C)\leq H_n(B,C)+\max\left\{{\rm diam}(B,d_n), {\rm diam}(C,d_n) \right\}.$$
\end{prop}
\begin{proof}
	Since $d_n\geq d$, $B\subset C_n^{\varepsilon}$ implies that $f^iB\subset(f^iC)^{\varepsilon}$, $i=0,1,\ldots,n-1$, therefore $H_n(B,C)\leq H^n(B,C)$.
	
	If $\varepsilon:=H_n(B,C)$, then for any $x\in B$ and $i=0,1,\ldots,n-1$, there exists $y_i\in C$ such that $d(f^ix,f^iy_i)\leq\varepsilon$. Then $$d(f^ix,f^iy_1)\leq d(f^ix,f^iy_i)+d(f^iy_1,f^iy_i)\leq\varepsilon+\text{diam}(C,d_n),\ i=0,1,\ldots,n-1,$$ which implies that $B\subset C_n^{\varepsilon+\text{diam}(C,d_n)}$. For the same reason $C\subset B_n^{\varepsilon+\text{diam}(B,d_n)}$.
\end{proof}

For $x\in X$, the $n$th empirical measure associated to $x$ is $$e_n^f(x):=\frac{1}{n}\sum_{i=0}^{n-1}\delta_{f^ix},$$ where $\delta_y$ denotes the Dirac measure at the point $y\in X$. Let $V(x)$ denote the set of limit points of the sequence $\{e_n^f \}_{n\in\mathbb{N}}$; it is always a compact connected nonempty subset of $\mathcal{M}_f(X)$. 
\begin{prop}
	The map defined by 
	\begin{align*}
	\varPhi:X&\to\mathcal{K}(\mathcal{M}(X))\\
	x&\mapsto V(x).
	\end{align*}
	is continuous.
\end{prop}
\begin{proof}
	For $\varepsilon>0$, we will show that $H(V(x),V(y))<\varepsilon
	$ when $x,y\in X$ close enough, by proving that for any $\mu\in V(x)$, there exists a $\nu\in V(y)$ such that $D(\mu,\nu)<\varepsilon$.
	
	Take $M\in\mathbb{N}$ such that $\sum_{n=M+1}^{\infty}\frac{1}{2^n}<{\varepsilon}/{4}$. 
	For $\mu\in V(x)$, there exists a sequence of positive integers $\{n_k\}_{k\geq 1}$ such that $n_k\to\infty$ and $e_{n_k}^f(x)\to\mu$ as $k\to\infty$. Consider $\{e_{n_k}^f(y) \}$. We may assume that $e_{n_k}^f(y)\to\nu\in V(y)$. Pick $n_{0}$ large enough such that $D(e_{n_{0}}^f(x),\mu)<{\varepsilon}/{4}$ and $D(e_{n_{0}}^f(y),\nu)<{\varepsilon}/{4}$. We can find $\zeta>0$ such that $|\varphi_i(f^jp)-\varphi_i(f^jq)|<{\varepsilon}/{4}$ for $i=1,\ldots, M$ and $j=0,1,\ldots,n_{0}-1$ whenever $d(p,q)<\zeta$. Therefore,
	\begin{align*}
		D(\mu,\nu)\leq& D(e_{n_{0}}^f(x),\mu)+D(e_{n_{0}}^f(y),\nu)+D(e_{n_{0}}^f(x),e_{n_{0}}^f(y))\\
		<&
		\dfrac{3\varepsilon}{4}+\sum_{n=1}^{M}\dfrac{|\int\varphi_n\,\mathrm{d}e_{n_{0}}^f(x)-\int\varphi_n\,\mathrm{d}e_{n_{0}}^f(y)|}{2^n}
		<\varepsilon.
	\end{align*}
\end{proof}

\section{Dynamical Emergence}
In this section, we define the dynamical emergences and entropy orders for $\mathcal{M}(X)$, $\mathcal{K}(X)$ and invariant measures.
\subsection{Dynamical emergence of $\mathcal{M}(X)$}
We first recall the definition of  topological entropy for a topological dynamical system $(X,f)$. A subset  $E\subset X$ is called a $(n,\varepsilon)$-spanning set if for every $x\in X$ there exists $y\in E$ such that $d_n(x,y)<\varepsilon$. A subset $F$ is called a  $(n,\varepsilon)$-separated set if $d_n(x,y)\geq\varepsilon$ for every $x,y\in F$ with $x\neq y$. Denote by $N(n,\varepsilon)$ the smallest cardinality of any $(n,\varepsilon)$-spanning set, and by $S(n,\varepsilon)$ the largest cardinality of any $(n,\varepsilon)$-separated set. The topological entropy, $h_{top}(X,f)$, is defined by
\begin{align*}
  h_{top}(f)&=\lim_{\varepsilon\to0}\limsup_{n\to\infty}\frac{\log N(n,\varepsilon)}{n}\\
  &=\lim_{\varepsilon\to0}\limsup_{n\to\infty}\frac{\log S(n,\varepsilon)}{n}.
\end{align*}
Especially, the {\rm limsup} can be replaced by {\rm liminf} in the above formula.

Let $(X,d)$ be a compact metric space and $f_{\mathcal{M}}$ the induced transformation.
We will define the dynamical emergence for an arbitrary nonempty subset $\mathcal{Z}\subset\mathcal{M}(X)$.

For $\varepsilon>0$, a set $E\subset \mathcal{M}(X)$ is called a $(n,\varepsilon)$-spanning set of $\mathcal{Z}$ if  for every $\mu\in \mathcal{Z}$, there exists $\nu\in E$ such that $W_1^n(\mu,\nu)\leq\varepsilon$. Further, $E\subset \mathcal{Z}$ is called a5 $(n,\varepsilon)$-separated set of $\mathcal{Z}$ if for distinct $\mu,\nu\in E$, we have $W_1^n(\mu,\nu)>\varepsilon$. Let $N_\mathcal{M}(\mathcal{Z},n,\varepsilon)$ denote the smallest cardinality of the $(n,\varepsilon)$-spanning sets for $\mathcal{Z}$, and let $S_\mathcal{M}(\mathcal{Z},n,\varepsilon)$ denote the largest cardinality of the $(n,\varepsilon)$-separated sets for $\mathcal{Z}$. It is easy to check that $N_\mathcal{M}(\mathcal{Z},n,\varepsilon)\leq S_\mathcal{M}(\mathcal{Z},n,\varepsilon)\leq N_\mathcal{M}(\mathcal{Z},n,\varepsilon/2)$. For simplicity, write $N_\mathcal{M}(n,\varepsilon)$ for $N_\mathcal{M}(\mathcal{M}(X),n,\varepsilon)$ and $S_\mathcal{M}(n,\varepsilon)$ for $S_\mathcal{M}(\mathcal{M}(X),n,\varepsilon)$.
\begin{defn}\rm\label{mem}
	For a nonempty subset $\mathcal{Z}\subset\mathcal{M}(X)$ and $\varepsilon>0$,
	\emph{the dynamical emergence} of $\mathcal{Z}$ with scale $\varepsilon$ is defined as the sequence $\left\{ N_\mathcal{M}(\mathcal{Z},n,\varepsilon) \right\}_{n\in\mathbb{N}}$.
	In addition, let
	$$\mathcal{E}(f_\mathcal{M},\mathcal{Z})=\lim_{\varepsilon\to0}\limsup_{n\to\infty}\frac{\log\log N_\mathcal{M}(\mathcal{Z},n,\varepsilon)}{n},$$ and call it the \emph{entropy order} of $\mathcal{Z}$ associated with $\{W_1^n \}$. Write $\mathcal{E}(f_\mathcal{M},\mathcal{M}(X))$ as $\mathcal{E}(f_\mathcal{M})$ for short.
\end{defn}
\begin{rem}
	{\rm (1)} Inequalities  $(2.1)$ and $(2.2)$ show that we can replace $W_1^n$ by $W_p^n$, $1<p<\infty$ and $LP^n$ in the definition of entropy order.
	
	\noindent{\rm (2)}  
	It is easy to check that $\mathcal{E}(f_\mathcal{M},\mathcal{Z})=\lim\limits_{\varepsilon\to0}\limsup\limits_{n\to\infty}\dfrac{\log\log S_\mathcal{M}(\mathcal{Z},n,\varepsilon)}{n}$.
\end{rem}

If we replace metric $W_1^n$ by the Bowen metrics on $(\mathcal{M}(X), f_\mathcal{M})$, $W_{p,n}$, $1\leq p<\infty$ and $LP_n$.
Since $\mathcal{M}(X)$ is compact, we have $\limsup\limits_{n\to\infty}\frac{1}{n}\log \tilde{N}_\mathcal{M}(\mathcal{Z},n,\varepsilon)<\infty$, where $\tilde{N}_\mathcal{M}(\mathcal{Z},n,\varepsilon)$ denotes the smallest cardinality of $\varepsilon$-spanning sets for $\mathcal{Z}\subset\mathcal{M}(X)$ with metric $W_{1,n}$. Hence for every $\varepsilon>0$ we have $\limsup\limits_{n\to\infty}\frac{1}{n}\log\log \tilde{N}_\mathcal{M}(\mathcal{Z},n,\varepsilon)=0$.
This demonstrates that the choice of metric in Definition 3.1 seems to be more suitable than Bowen metric on $\mathcal{M}(X)$ for the study of emergence.

A continuous map $g:Y\to Y$ of a compact metric space $Y$ is said to be a factor of $f:X\to X$ if there exists a continuous map $\phi$ from $X$ onto $Y$ satisfying $g\circ\phi=\phi\circ f$. Moreover, if $\phi$ is a homeomorphism, we say $g$ is topologically conjugate to $f$. It is well known that $h_{top}(f)\geq h_{top}(g)$ if $g$ is a factor of $f$ and $h_{top}(f)=h_{top}(g)$ if $g$ is topologically conjugate to $f$ (see \cite{Walters1982} for a proof).   Note that if $g$ is a factor of $f$, then $\phi_{\ast}:\mathcal{M}(X)\to\mathcal{M}(Y)$, defined by $\phi_{\ast}\mu=\mu\circ\phi^{-1}$, is a continuous surjective map, and satisfies $\phi_{\ast}\circ f_{\mathcal{M}}=g_{\mathcal{M}}\circ\phi_{\ast}$.

Now, we demonstrate some properties of $\mathcal{E}(f_\mathcal{M},\cdot)$ as follows:
\begin{prop}\label{31}
	{\rm (1)} If $\mathcal{Z}'\subset \mathcal{Z}\subset \mathcal{M}(X)$, then $\mathcal{E}(f_{\mathcal{M}},\mathcal{Z}')\leq \mathcal{E}(f_{\mathcal{M}},\mathcal{Z})$.\\
	{\rm (2)} If $g:Y\to Y$ is a factor of $f:X\to X$, then $\mathcal{E}(f_{\mathcal{M}})\geq \mathcal{E}(g_{\mathcal{M}})$. In particular, if $f$ and $g$ are topologically conjugate then $\mathcal{E}(f_{\mathcal{M}})= \mathcal{E}(g_{\mathcal{M}})$.\\
	{\rm (3)} $\mathcal{E}(f_{\mathcal{M}},\mathcal{Z})=\mathcal{E}(f_{\mathcal{M}},\overline{\mathcal{Z}})$, where $\overline{\mathcal{Z}}$ denotes the closure of $\mathcal{Z}$.
\end{prop}
\begin{proof}
	(1) Obvious.\\
	(2) This is a simple matter, which is clear from Theorem \ref{mh}.\\
	(3) It is sufficient to show $\mathcal{E}(f_{\mathcal{M}},\mathcal{Z})\geq \mathcal{E}(f_{\mathcal{M}},\overline{\mathcal{Z}})$. Denote by $B(\mu,n,\varepsilon)$ the $\varepsilon$-neighbourhood of $\mu\in\mathcal{M}(X)$ with metric $W_1^n$. Then $\overline{\mathcal{Z}}\subset \cup_{\mu\in E}B(\mu,n,2\varepsilon)$ provided by $\mathcal{Z}\subset \cup_{\mu\in E}B(\mu,n,\varepsilon)$,
	and therefore $N_\mathcal{M}(\overline{\mathcal{Z}},n,2\varepsilon)\leq N_\mathcal{M}(\mathcal{Z},n,\varepsilon)$.
\end{proof}

\subsection{Dynamical emergence of a measure}
   For $\mu\in\mathcal{M}_f(X)$, according to the Birkhoff ergodic theorem, for $\mu$-a.e. $x\in X$, the sequence $\{e^f_n(x) \}_{n\in\mathbb{N}}$ converges to a unique measure $e^f(x)\in\mathcal{M}_f(X)$.
\begin{defn}\rm\label{eem}
	For $\mu\in\mathcal{M}_f(X)$, $n\in\mathbb{N}$ and $\varepsilon>0$, define $$\mathcal{E}_\mu(n,\varepsilon):=\min\left\{ N:\exists\mathcal{F}=\{\mu_1,\ldots,\mu_N \}\subset\mathcal{M}(X)\ s.\ t. \int W_1^n(e^f(x),\mathcal{F})\,\mathrm{d}\mu(x)\leq\varepsilon \right\}.$$
	Call the sequence $\left\{ \mathcal{E}_\mu(n,\varepsilon
	) \right\}_{n\in\mathbb{N}}$ the
	\emph{dynamical emergence} of $\mu$ with scale $\varepsilon$. Set
	$$\mathcal{E}_\mu=\lim_{\varepsilon\to0}\limsup_{n\to\infty}\frac{\log\log \mathcal{E}_\mu(n,\varepsilon)}{n},$$ and call it the \emph{entropy order} of $\mu$. In particular, we set $\log0=0.$
\end{defn}
\begin{rem}
	For $\mu\in\mathcal{M}_f^{erg}(X)$, we have $e^f(x)=\mu$ for $\mu$-a.e. $x\in X$, therefore $\mathcal{E}_\mu(n,\varepsilon)=1$ for any $n\in\mathbb{N}$, $\varepsilon>0$, hence $\mathcal{E}_\mu=0$ for each ergodic invariant measure.
\end{rem}
To give the proof of Theorem \ref{variation}, we need the following concept, which is the dynamical version of the quantization number of a measure (see \cite{Graf2000}). Denote by $\mathcal{W}^n$ the $1$-Wasserstein metric on $\mathcal{M}(\mathcal{M}(X))$ induced by $W_1^n$.
\begin{defn}\rm
	For $\omega\in\mathcal{M}(\mathcal{M}(X))$ and $n\in\mathbb{N}$, $\varepsilon>0$, define
	$${Q}(\omega)=\lim_{\varepsilon\to0}\limsup_{n\to\infty}\frac{\log\log Q(\omega,n,\varepsilon)}{n},$$ where
	\begin{align*}
	Q(\omega,n,\varepsilon)=\min \Big\{
	 N:\ &\text{there exists a probability measure}\ \rho\in\mathcal{M}(\mathcal{M}(X))\\
	&\text{supported on a set of cardinality}\ N,\ \text{such that}\ \mathcal{W}^n(\omega,\rho)\leq\varepsilon \Big\}.
	\end{align*}
\end{defn}
\begin{prop}
	$Q(\omega,n,\varepsilon)=\min\Big\{N\in\mathbb{N}: \exists\mathcal{F}=\{\mu_1,\cdots,\mu_N\}
	\subset\mathcal{M}(X)\ \text{such that} \\  \int W_1^n(\mu, \mathcal{F}) \,\mathrm{d}\omega(\mu)\leq\varepsilon\Big\}.$
\end{prop}
\begin{proof}
	Let $\mathcal{F}\subset\mathcal{M}(X)$ be a set of minimal cardinality $N$ such that $\int W_1^n(\mu,\mathcal{F})\,\mathrm{d}\omega(\mu)\leq\varepsilon$. Take a measurable map $h:\mathcal{M}(X)\to\mathcal{F}$ satisfying $W_1^n(\mu,h(\mu))=W_1^n(\mu,\mathcal{F})$ for all $\mu\in\mathcal{M}(X)$. Let $\rho=h_\ast\omega$. Then $\text{supp}\rho\subset\mathcal{F}$. Consider the map $\text{id}\times h:\mathcal{M}(X)\to\mathcal{M}(X)\times\mathcal{F}$, $\mu\mapsto(\mu,h(\mu))$, and let $\pi=(\text{id}\times h)_\ast\omega$. It is easy to check that $\pi\in\Pi(\omega,\rho)$ and $\int W_1^n(\mu,\nu)\,\mathrm{d}\pi(\mu,\nu)\leq\varepsilon$. Thus we have $D^n(\omega,\rho)\leq\varepsilon$, i.e., $Q(\omega,n,\varepsilon)\leq N$.
	
	Let $\rho\in\mathcal{M}(\mathcal{M}(X))$ be a measure such that $D^n(\omega,\rho)\leq\varepsilon$, $\text{supp}\rho=\mathcal{F}$ and $\#\mathcal{F}=Q(\omega,n,\varepsilon)$. This means that there exists $\pi\in\Pi(\omega,\rho)$ such that $\int W_1^n(\mu,\nu)\,\mathrm{d}\pi(\mu,\nu)\leq\varepsilon$. Consider a disintegration of $\pi$, that is a collection of measures $\nu_{\xi}\in\mathcal{M}(\mathcal{M}(X))$ for $\omega$-a.e. $\xi\in\mathcal{M}(X)$, such that $\pi=\int\delta_\xi\otimes\nu_\xi \,\mathrm{d}\omega(\xi)$. Since $\rho$ is one of the marginals of $\pi$, we have $\text{supp}\nu_\xi\subset\mathcal{F}$ for $\omega$-a.e. $\xi$. Therefore,
	$$\int W_1^n(\xi,\mathcal{F})\,\mathrm{d}\omega(\xi)\leq \int\int W_1^n(\xi,\eta)\,\mathrm{d}\nu_\xi(\eta)\,\mathrm{d}\omega(\xi)\leq\varepsilon,$$
	which completes the proof.
\end{proof}

\begin{prop}
	For any $\omega\in\mathcal{M}(\mathcal{M}_f^{erg}(X))$, $Q(\omega,n,\varepsilon)\leq N_\mathcal{M}(\mathcal{M}_f^{erg}(X),n,\varepsilon)$.
\end{prop}
\begin{proof}
	Given a $(n,\varepsilon)$-spanning set of $\mathcal{M}_f^{erg}(X)$, we can transport any $\omega\in\mathcal{M}(\mathcal{M}_f^{erg}(X))$ to a measure supported on it with a cost less than $\varepsilon$, which shows that $Q(\omega,n,\varepsilon)\leq N_\mathcal{M}(\mathcal{M}_f^{erg}(X),n,\varepsilon)$.
\end{proof}

A measure $\omega\in\mathcal{M}(\mathcal{M}_f^{erg}(X))$ is said to be the ergodic decomposition of $\mu\in\mathcal{M}_f(X)$ if $\mu=\int\eta \,\mathrm{d}\omega(\eta)$.
A trivial verification from the definitions shows that if $\omega$ is the ergodic decomposition of $\mu\in\mathcal{M}_f(X)$, then $Q(\omega,n,\varepsilon)=\mathcal{E}_\mu(n,\varepsilon)$, which implies the following proposition.
\begin{prop}\label{xiaoyu}
	For any $\mu\in\mathcal{M}_f(X)$, we have $\mathcal{E}_\mu\leq \mathcal{E}(f_{\mathcal{M}},\mathcal{M}_f^{erg}(X))$.
\end{prop}

\subsection{Dynamical emergence of $\mathcal{K}(X)$}
Similar to $\mathcal{E}(f_{\mathcal{M}})$, a set $\mathcal{G}\subset\mathcal{K}(X)$ is called a $(n,\varepsilon)$-spanning set of $\mathcal{Z}\subset\mathcal{K}(X)$ if for every $B\in\mathcal{Z}$, there exists $C\in\mathcal{G}$ such that $H^n(B,C)\leq\varepsilon$. A set $\mathcal{G}\subset\mathcal{Z}$ is called a $(n,\varepsilon)$-separated set if for distinct $B,C\in \mathcal{G}$, we have $H^n(B,C)>\varepsilon$. Denote by $N_\mathcal{K}(\mathcal{Z},n,\varepsilon)$ the smallest cardinality of the $(n,\varepsilon)$-spanning sets of $\mathcal{Z}$, and by $S_\mathcal{K}(\mathcal{Z},n,\varepsilon)$ the largest cardinality of the $(n,\varepsilon)$-separated sets of $\mathcal{Z}$. For simplicity, we will write $N_\mathcal{K}(n,\varepsilon)$ for $N_\mathcal{K}(\mathcal{K}(X),n,\varepsilon)$ and $S_\mathcal{K}(n,\varepsilon)$ for $S_\mathcal{K}(\mathcal{K}(X),n,\varepsilon)$.
\begin{defn}\rm\label{kme}
	For a  nonempty subset $\mathcal{Z}\subset\mathcal{K}(X)$ and $\varepsilon>0$, the \emph{dynamical emergence} of $\mathcal{Z}$ with scale $\varepsilon$ is defined as the sequence $\left\{ N_\mathcal{K}(\mathcal{Z},n,\varepsilon) \right\}_{n\in\mathbb{N}}$.
	In addition, let
	$$\mathcal{E}(f_\mathcal{K},\mathcal{Z})=\lim_{\varepsilon\to0}\limsup_{n\to\infty}\frac{\log\log N_\mathcal{K}(\mathcal{Z},n,\varepsilon)}{n},$$ and call it the  \emph{entropy order} of $\mathcal{Z}$.
	For simplicity, we will write $\mathcal{E}(f_{\mathcal{K}})$ for $\mathcal{E}(f_{\mathcal{K}},\mathcal{K}(X))$.
\end{defn}
\begin{rem}
	{\rm(1)}
	Actually according to the proofs of main results, we can replace {\rm limsup} by {\rm liminf} in the definitions of $\mathcal{E}(f_{\mathcal{M}})$ and $\mathcal{E}(f_{\mathcal{K}}),$ see Remark {\rm\ref{inf}}.\\
	\noindent{\rm(2)}
	It follows easily that $\mathcal{E}(f_\mathcal{K},\mathcal{Z})=\lim\limits_{\varepsilon\to0}\limsup\limits_{n\to\infty}\dfrac{\log\log S_\mathcal{K}(\mathcal{Z},n,\varepsilon)}{n}$.
\end{rem}

If $g:Y\to Y$ is a factor of $f:X\to X$ with factor map $\phi:X\to Y$, then $g_{\mathcal{K}}:\mathcal{K}(Y)\to\mathcal{K}(Y)$ is also a factor of $f_{\mathcal{K}}:\mathcal{K}(X)\to\mathcal{K}(X)$ with factor map $B\mapsto\phi(B)$, which we denote by $\phi_{\mathcal{K}}$.
The following results may be proved in  the same way as Proposition \ref{31}.

\begin{prop}
	{\rm (1)} If $\mathcal{Z}'\subset \mathcal{Z}\subset \mathcal{K}(X)$, then $\mathcal{E}(f_{\mathcal{K}},\mathcal{Z}')\leq \mathcal{E}(f_{\mathcal{K}},\mathcal{Z})$.\\
	{\rm (2)} If $g:Y\to Y$ is a factor of $f:X\to X$, then $\mathcal{E}(f_{\mathcal{K}})\geq \mathcal{E}(g_{\mathcal{K}})$. In particular, if $f$ and $g$ are topologically conjugate then $\mathcal{E}(f_{\mathcal{K}})= \mathcal{E}(g_{\mathcal{K}})$.\\
	{\rm (3)} $\mathcal{E}(f_{\mathcal{K}},\mathcal{Z})=\mathcal{E}(f_{\mathcal{K}},\overline{\mathcal{Z}})$, where $\overline{\mathcal{Z}}$ denotes the closure of $\mathcal{Z}$.
\end{prop}

\section{Proofs of the main results}
In this section, we give the proofs of our main results. We shall denote $\lfloor a\rfloor$ the largest integer smaller than $a\in\mathbb{R}$.
\subsection{Proof of Theorem \ref{mh}}

The following formula comes from \cite[Theorem A.1]{Bolley2007}, which gives an upper bound on the number of the balls in Wasserstein distance needed to cover the space $\mathcal{M}(X)$.
\begin{thm}
	Let $(X,d)$ be a Polish space with finite diameter $D$. For any $r>0$, define $N(X,r)$ as the minimal number of balls needed to cover $X$ by balls of radius $r$. Then, for all $p\geq 1$ and $\delta\in (0,D)$, the space $\mathcal{M}(X)$ can be covered by $\mathcal{N}_p(X,\delta)$ balls of radius $\delta$ in $W_p$ distance, with $$\mathcal{N}_p(X,\delta)\leq\left( \frac{8eD}{\delta} \right)^{pN\left(X,\frac{\delta}{2}\right)}.$$
\end{thm}
According to above theorem, we have for any $\varepsilon>0$, there exists a constant $C>0$ such that $$N_\mathcal{M}(n,\varepsilon )\leq
\big( C/\varepsilon \big)^{N(f,\,n,\,{\varepsilon}/{2})}.$$
So,
\begin{align*}
\frac{\log\log N_\mathcal{M}(n,\varepsilon)}{n}
\leq
\frac{\log N(f,n,\varepsilon/2)}{n}
+\frac{\log\log(C/\varepsilon)}{n},
\end{align*}
and hence
$$\limsup_{n\to\infty}\frac{\log\log N_\mathcal{M}(n,\varepsilon)}{n}
\leq
\limsup_{n\to\infty}\frac{\log N(f,n,\varepsilon/2)}{n}.$$
Now taking the limit as $\varepsilon\to0$, we obtain $\mathcal{E}(f_{\mathcal{M}})\leq  h_{top}(f)$.

Now we turn to prove $\mathcal{E}(f_{\mathcal{M}})\geq h_{top}(f)$, our proof is adapted from \cite[Theorem 1.6]{Berger2019}.

For $n\in\mathbb{N}$ and $\varepsilon>0$, we say that $\mu,\nu\in\mathcal{M}(X)$ are $(n,\varepsilon)$-apart if $$\min\left\{d_n(x,y): x\in {\rm supp}(\mu), y\in {\rm supp}(\nu)\right\}\geq\varepsilon.$$ A nonempty subset $\mathcal{A}\subset\mathcal{M}(X)$ is called convex if for any $\mu_1,\ldots,\mu_n\in\mathcal{A}$, we have $\sum_{\i=1}^na_i\mu_i\in\mathcal{A}$, where $0\leq a_i\leq1$ for $1\leq i\leq n$ and $\sum_{i=1}^na_i=1$. Denote by $A(\mathcal{A},n,\varepsilon)$ the maximal number of pairwise $(n,\varepsilon)$-apart measures in $\mathcal{A}\subset\mathcal{M}(X)$ and, for simplicity, write $A(n,\varepsilon)$ for $A(\mathcal{M}(X),n,\varepsilon)$ .

\begin{lem}{\rm(Berstein inequality \cite{Grimmett2001})}\label{berstein}
	Let $H_n$ be the number of heads in $n$ tosses of a fair coin. Then for any $\delta>0$, $${\rm Prob}\Big(\frac{H_n}{n}\leq\frac{1}{2}-\delta\Big)\leq e^{-\frac{\pi}{4}\delta^2n}.$$
\end{lem}
\begin{lem}\label{ky}
	Let $\mathcal{A}$ be a convex subset of $\mathcal{M}(X)$. Then $$\mathcal{E}(f_{\mathcal{M}},\mathcal{A})\geq\lim_{\varepsilon\to0}\limsup_{n\to\infty}\frac{\log A(\mathcal{A},n,\varepsilon)}{n}.$$
\end{lem}
\begin{proof}
	For $\varepsilon>0$ and $n\in\mathbb{N}$, let $N:=8\lfloor \frac{A(\mathcal{A},n,\varepsilon)}{8} \rfloor$. Then we have $A(\mathcal{A},n,\varepsilon)-7\leq N\leq A(\mathcal{A},n,\varepsilon)$, so we can find $\nu_1,\ldots,\nu_N\in\mathcal{A}$ that are pairwise $(n,\varepsilon)$-apart. Define $$F:=\left\{ \phi:\{1£¬\ldots, N\}\to\{0,1\}:\sum_{i=1}^N \phi(i)=\frac{N}{2}\right\}$$ endowed with the Hamming distance $${\rm Hamm}(\phi_1,\phi_2):=\#\big\{i\in\{1,\ldots,N\}:\phi_1(i)\neq\phi_2(i) \big\}.$$ Note that the distance is always an even number and $\# F=\dbinom{N}{N/2}\geq (2N)^{-\frac{1}{2}}2^N$ (by Stirling's formula). Let $\phi\in F$ and $U$ be the $\frac{N}{4}$-neighbourhood of $\phi$.
	For $\phi'\in U$, set $k:={\rm Hamm}(\phi,\phi')\leq\frac{N}{8}$. Then there exist $k$ elements of $\phi^{-1}(\{1\})$ and $k$ elements of $\phi^{-1}(\{0\})$ at which $\phi$
	differs from $\phi'$, i.e., $\phi'$ corresponds to $k$ elements in $\phi^{-1}(\{1\})$ and $k$ elements in $\phi^{-1}(\{0\})$. On the other hand, $\#\phi^{-1}(\{1\})=\#\phi^{-1}(\{0\})=\frac{N}{2}$. Therefore,
	$$\# U=\sum_{k=0}^{N/8}\dbinom{N/2}{k}^2\leq  \left[\sum_{k=0}^{N/8}\dbinom{N/2}{k}\right]^2.$$
	
	Moreover, $\sum_{k=0}^{N/8}\dbinom{N/2}{k}$ equals $2^{N/2}$ times the probability of the number of heads in $\frac{N}{2}$ tosses of a fair coin is less than $\frac{N}{8}$. According to Lemma \ref{berstein}, we have
	$$\# U\leq\left[\sum_{k=0}^{N/8}\dbinom{N/2}{k}\right]^2\leq \big(2^{N/2}e^{-\frac{\pi}{4}\frac{N}{16}\frac{1}{2}} \big)^2=2^Ne^{-\frac{\pi}{4}\frac{N}{16}}.$$
	
	Let $F'$ be a $\frac{N}{4}$-separated set of $F$ with maximal cardinality. Then $$\# F'\geq\frac{\# F}{\# U}\geq(2N)^{-\frac{1}{2}}2^N2^{-N}e^{\frac{\pi}{4}\frac{N}{16}}=(2N)^{-\frac{1}{2}}e^{\pi\frac{N}{64}},$$
	and, for large enough $N$, we have $\# F'\geq e^{cN}$, where $c>0$ is a constant.

	For each $\phi\in F'$, consider the measure $$\mu_{\phi}:=\frac{2}{N}\sum_{i=1}^N\phi(i)\nu_i.$$ Let $\mathcal{F}:=\{\mu_{\phi}:\phi\in F' \}.$ Note that $\mathcal{F}\subset\mathcal{A}$ as $\mathcal{A}$ is convex.  For two distinct elements $\phi_1,\phi_2\in F'$, let $S_1$ and $S_2$ be the support sets of $\mu_{\phi_2}$ and $\mu_{\phi_2}$ respectively.
	For $(x,y)\in(S_1\setminus S_2)\times S_2$, there exists $j\in\{1,\ldots,N \}$ such that $\phi_2(j)=1$ and $\phi_1(j)=0$. So $d_n(x,y)\geq\varepsilon$ as $\nu_1,\ldots,\nu_n$ are pairwise $(n,\varepsilon)$-apart. Moreover,
	$$\mu_{\phi_1}(S_1\setminus S_2)=\frac{2}{N}\#\big\{i\in\{1,\ldots,N\}:\phi_1(i)=1,\phi_2(i)=0 \big\}=\frac{{\rm Hamm}(\phi_1,\phi_2)}{N}\geq\frac{1}{4}.$$ Therefore for any $\pi\in\Pi(\mu_{\phi_1},\mu_{\phi_2})$,
	\begin{align*}
	\int_{X\times X}d_n(x,y)d\pi(x,y)&\geq \int_{(S_1\setminus S_2)\times S_2}d_n(x,y)d\pi(x,y)\\
	&\geq\varepsilon \pi[(S_1\setminus S_2)\times S_2]=\varepsilon\mu_{\phi_1}(S_1\setminus S_2)\geq\frac{\varepsilon}{4}.
	\end{align*}
	We conclude that $W_1^n(\mu_{\phi_1},\mu_{\phi_2})\geq\frac{\varepsilon}{4}$, i.e., $\mathcal{F}$ is a $(n,\varepsilon/4)$-separated set. So
	$$S_\mathcal{M}(\mathcal{A},n,\frac{\varepsilon}{4})\geq e^{cN},$$
	and therefore,
	\begin{equation}\label{minf}
		\frac{\log\log S_\mathcal{M}(\mathcal{A},n,\frac{\varepsilon}{4})}{n}\geq \frac{\log c+\log (A(\mathcal{A},n,\varepsilon)-7)}{n}.
	\end{equation}
	Now taking the limsup as $n\to\infty$ and the limit as $\varepsilon\to0$, which completes the proof.
\end{proof}

Observe that if $\{x_1,\ldots,x_l\}$ is $(n,\varepsilon)$-separated, then $\{\delta_{x_1},\ldots,\delta_{x_l}\}$ is pairwise $(n,\varepsilon)$-apart, and therefore, $$S(f,n,\varepsilon)\leq A(n,\varepsilon).$$  Let $\mathcal{A}$ be $\mathcal{M}(X)$ in Lemma \ref{ky}, we get $\mathcal{E}(f_{\mathcal{M}})\geq h_{top}(f).$

\subsection{Proof of Theorem \ref{expansive}}
It suffices to prove that $\mathcal{E}(f_{\mathcal{M}},\mathcal{M}_f^{per}(X))\geq h_{top}(f)$.
Let $\rho$ be the expansive constant of $f$. For $\delta>0$, we have
$$S(f,n,\varepsilon)\geq e^{n(h_{top}(f)-\delta)}$$
for some $\varepsilon>0$ and large enough $n\in\mathbb{N}$. 
Let $F_n\subset X$ be an $(n,\varepsilon)$-separated set with $\# F_n=S(f,n,\varepsilon)$. 
According to specification property and Lemma \ref{sp+exp}, there exists $n_0$ (depending on $\varepsilon$) such that every $x\in F_n$ is shadowed by an $(n+n_0)$-periodic point $y\in X$ in the sense that $d_n(x,y)<\varepsilon/2$. Denote by $G_n$ the set of periodic points obtained in this way. Then $\# G_n=\# F_n$.

Let $\Pi_n=\{f^k(y):y\in G_n,k\in\mathbb{N} \}$ be the union of orbits of the points in $G_n$.
For distinct $y,z\in\Pi_n$, by expansiveness and periodicity, there exists $0\leq k<n+n_0$ such that $d(f^k(y),f^k(z))>\rho$, i.e., $y,z$ are $(n+n_0,\rho)$-separated. So any two distinct ergodic invariant measures supported in $\Pi_n$ are $(n+n_0,\rho)$-apart. On the other hand, the number of ergodic measures supported in $\Pi_n$, denoted by $A_n$, satisfies $$A_n\geq\frac{\# G_n}{n+n_0}=\frac{\# F_n}{n+n_0}\geq e^{n(h_{top}(f)-2\delta)},$$ for large enough $n$.
Moreover, we have $$A(\mathcal{M}_f(X),n+n_0,\rho)\geq A_n\geq e^{n(h_{top}(f)-2\delta)}.$$
Therefore,
$$\frac{\log A(\mathcal{M}_f(X),n+n_0,\rho)}{n}\geq h_{top}(f)-2\delta.$$
Letting $n\to\infty$ and noting that $A(\mathcal{M}_f(X),n,\eta)$ increases as $\eta\to0$, we obtain
$$\lim_{\eta\to0}\limsup_{n\to\infty}\frac{\log A(\mathcal{M}_f(X),n,\eta)}{n}\geq h_{top}(f)-2\delta.$$
Combining Lemma \ref{ky} and the fact that $\delta$ was arbitrary,
$$\mathcal{E}(f_{\mathcal{M}},\mathcal{M}_f(X))\geq h_{top}(f).$$

Note that $\overline{\mathcal{M}_f^{per}(X)}=\mathcal{M}_f(X)$ (Theorem 1 in \cite{Sigmund1974}). So, according to Proposition \ref{31} (3), we conclude that $$\mathcal{E}(f_{\mathcal{M}},\mathcal{M}_f^{per}(X))\geq h_{top}(f),$$
which completes the proof.

\subsection{Proof of Theorem \ref{K}}
For $n\in\mathbb{N}$ and $\varepsilon>0$, let $E=\{x_1,\ldots,x_{N(f,\,n,\,\varepsilon)} \}\subset X$ be the $(n,\varepsilon)$-spanning set with the smallest cardinality. Set $\mathcal{G}=2^E\setminus\{\emptyset \}\subset\mathcal{K}(X)$, i.e., the collection of all nonempty subsets of $E$. Then $\mathcal{G}$ is a $(n,\varepsilon)$-spanning set for $f_{\mathcal{K}}$. Indeed, for any $B\in\mathcal{K}(X)$, if $C=\{x\in E: B\cap B_{n}(x,\varepsilon)\neq\emptyset \}\in\mathcal{G},$ then $H^n(B,C)\leq\varepsilon$. This implies $N_\mathcal{K}(n,\varepsilon)\leq 2^{N(f,\,n,\,\varepsilon)},$ thus
$$\frac{\log\log N_\mathcal{K}(n,\varepsilon)}{n}\leq
\frac{\log N(f,n,\varepsilon)+\log\log2}{n},$$
and so we conclude that $\mathcal{E}(f_{\mathcal{K}})\leq h_{top}(f)$.

On the other hand, let $N:=8\lfloor\frac{S(f,\,n,\,\varepsilon)}{8} \rfloor$. Then $S(f,n,\varepsilon)-7\leq N\leq S(f,n,\varepsilon)$. We can find a set $E=\{x_1,\ldots,x_N \}\subset X$
is $(n,\varepsilon)$-separated. Let $F$ and $F'$ be  the sets in the proof of  Lemma \ref{ky}, where $F'$ is a $\frac{N}{4}$-separated set of $F$ with $\# F\geq e^{cN}$ for large enough $N$ and $c>0$ is a constant.

For each $\phi\in F'$, write $B_{\phi}=\{x_i:\phi(i)=1 \}\in\mathcal{K}(X)$ and $\mathcal{B}=\{B_{\phi}:\phi\in F' \}\subset \mathcal{K}(X)$. It is easy to check that $\mathcal{B}$ is a $(n,\varepsilon)$-separated set for $f_{\mathcal{K}}$. Then $S_\mathcal{K}(n,\varepsilon)\geq e^{cN}\geq e^{c(S(f,\,n,\,\varepsilon)-7)}$, and thus
\begin{equation}\label{kinf}
	\dfrac{\log\log S_\mathcal{K}(n,\varepsilon)}{n}\geq
	\dfrac{\log c+\log(S(f,n,\varepsilon)-7)}{n}.
\end{equation}
Now we can conclude that $\mathcal{E}(f_{\mathcal{K}})\geq h_{top}(f)$.
\begin{rem}\label{inf}
	{\rm (1)}
	From been proved, formulas {\rm(\ref{minf})} and {\rm(\ref{kinf})} hold for arbitrary $n\in\mathbb{N}$ large enough, therefore we have
	\begin{align*}
	\mathcal{E}(f_{\mathcal{M}})&=\lim_{\varepsilon\to0}\liminf_{n\to\infty}\dfrac{\log\log N_\mathcal{M}(n,\varepsilon)}{n}=\lim_{\varepsilon\to0}\liminf_{n\to\infty}\dfrac{\log\log S_\mathcal{M}(n,\varepsilon)}{n},
	\\
	\mathcal{E}(f_{\mathcal{K}})
	&=
	\lim_{\varepsilon\to0}\liminf_{n\to\infty}\frac{\log\log N_\mathcal{K}(n,\varepsilon)}{n}
	=
	\lim_{\varepsilon\to0}\liminf_{n\to\infty}\frac{\log\log S_\mathcal{K}(n,\varepsilon)}{n}.
	\end{align*}\\
	{\rm (2)} Consider the mean metrics for $(X,f)$, which defined by 
	$$\hat{d}_n(x,y)=\frac{1}{n}\sum_{i=0}^{n-1}d(f^ix,f^iy).$$
	It was shown by Gr\"{o}ger and J\"{a}ger {\rm\cite{Groger2016}} that the topological entropy defined by mean metrics is equivalent to the topological entropy defined by Bowen metrics. Indeed we can apply our argument again, Theorem 1.2 and 1.3 still hold with Bowen metrics replaced by mean metrics in {\rm(\ref{wn})} and {\rm(\ref{hn})}.
\end{rem}

\subsection{Proof of Theorem \ref{Ex}}

We say a nonempty subset $\mathcal{A}\subset\mathcal{K}(X)$ satisfies the finite union property if for $B, C\in\mathcal{A}$, we have $B\cup C\in\mathcal{A}$. For $n\in\mathbb{N}$ and $\varepsilon>0$, we say that $B, C\in\mathcal{K}(X)$ are $(n,\varepsilon)$-split if $$\min\{d_n(x,y) :x\in B, y\in C\}>\varepsilon,$$ i.e., $d_n(B,C)>\varepsilon$. Denote by $B(\mathcal{A},n,\varepsilon)$ the maximal number of pairwise $(n,\varepsilon)$-split elements in $\mathcal{A}\subset\mathcal{K}(X)$.
\begin{lem}
	Let $(X,f)$ be a topological dynamical system and $\mathcal{A}$ a subset of $\mathcal{K}(X)$ satisfying the finite union property. Then we have $$\mathcal{E}(f_{\mathcal{K}},\mathcal{A})\geq\lim_{\varepsilon\to0}\limsup_{n\to\infty}\frac{\log B(\mathcal{A},n,\varepsilon)}{n}.$$
\end{lem}
\begin{proof}
	Fix $n\in\mathbb{N}$ and $\varepsilon>0$. Let $N:=8\lfloor\frac{B(\mathcal{A},\,n,\,\varepsilon)}{8}\rfloor$ and choose $B_1,\ldots,B_{N}\in\mathcal{A}$\ that are pointwise $(n,\varepsilon)$-split. Let $F$ and $F'$ be  the sets in the proof of  Lemma \ref{ky} with $\# F\geq e^{cN}$ for large enough $N$ and $c>0$ is a constant.
	
	For any $\phi\in F$, set $B_{\phi}=\cup_{\phi(i)=1}B_i\in\mathcal{A}$ and $\mathcal{B}=\{B_{\phi}:\phi\in F' \}$. For distinct $\phi_1,\phi_2\in F'$, there exists an $i$ such that $B_i\subset B_{\phi_1}\setminus B_{\phi_2}$. Then $d_n(B_j,B_i)>\varepsilon$ for any $j$ with $\phi_2(j)=1$. This implies that $B_{\phi_1}\nsubseteq (B_{\phi_1})_n^{\varepsilon}$, therefore $\mathcal{B}\subset\mathcal{K}(X)$ is a $(n,\varepsilon)$-separated set of $\mathcal{A}$. Then $S_\mathcal{K}(\mathcal{A},n,\varepsilon)\geq e^{cN}$, which completes the proof.
\end{proof}
We omit the rest of the proof, which from this point follows that of Theorem 1.4 with small changes, noting that $\mathcal{K}_f(X)\subset\mathcal{K}(X)$ satisfies the finite union property.

\subsection{Proof of Theorem \ref{variation}}
By Proposition \ref{xiaoyu}, it is sufficient to show the existence of a measure $\mu\in\mathcal{M}_f(X)$ such that $\mathcal{E}_\mu\geq \mathcal{E}(f_{\mathcal{M}},\mathcal{M}_f^{erg}(X)).$

Fix $\varepsilon>0$ and $n\in\mathbb{N}$, let $\mathcal{F}_n$ be the $(n,4\varepsilon)$ separated set of $\mathcal{M}_f^{erg}(X)$ with the largest cardinality, $c_n:=\#\mathcal{F}_n=S_\mathcal{M}(\mathcal{M}_f^{erg}(X),n,4\varepsilon)$. Denote by $\omega_n\in\mathcal{M}(\mathcal{M}_f^{erg}(X))$ the equidistributed measure with support on $\mathcal{F}_n$, i.e., $$\omega_n=\frac{1}{c_n}\sum_{\mu\in\mathcal{F}_n}\delta_\mu.$$
\begin{lem}{\rm \cite[Lemma 3.19]{Berger2019}}
    Let $\rho\in\mathcal{M}(\mathcal{M}(X))$ be any probability measure whose support has cardinality $m_n:=\lceil \frac{c_n}{2} \rceil$. Then $$\mathcal{W}^n(\omega_n,\rho)\geq\frac{c_n-m_n+1}{c_n}\cdot\frac{4\varepsilon}{2}.$$
\end{lem}
Therefore, $Q(\omega_n,n,\varepsilon)\geq m_n$. Let $\omega=\sum\limits_{n=1}^{\infty}2^{-n}\omega_n\in\mathcal{M}(\mathcal{M}_f^{erg}(X))$. Then $\omega\gg\omega_n$ for every $n\in\mathbb{N}$.
\begin{lem}{\rm\cite[Lemma 3.18]{Berger2019}}
	Let $\rho, \rho_1\in\mathcal{M}(\mathcal{M}_f^{erg}(X))$ be such that $\rho\gg\rho_1$. Then for any $n\in\mathbb{N}$ and $\varepsilon>0$, $$Q(\rho,n,\varepsilon)\geq Q(\rho_1,n,\varepsilon).$$
\end{lem}
Thus we have $$Q(\omega,n,\varepsilon)\geq\lceil \frac{c_n}{2} \rceil,$$ which implies that $Q(\omega)\geq \mathcal{E}(f_{\mathcal{M}},\mathcal{M}_f^{erg})$. Let $\mu=\int_{\mathcal{M}(X)}\nu \,\mathrm{d}\omega(\nu)$. It is easy to check that $\mu\in\mathcal{M}_f(X)$ and its ergodic decomposition is $\omega$. This clearly forces $\mathcal{E}_\mu\geq \mathcal{E}(f_{\mathcal{M}},\mathcal{M}_f^{erg})$.

\subsection{Proof of Theorem \ref{conformal}}

\begin{prop}\label{dimmo}
	For any compact metric space $(X,d)$, we have 
	$$\underline{{\rm dim}}(X,d)
	\leq
	\underline{{\rm mo}}(\mathcal{K}(X), H)
	\leq
	\overline{{\rm mo}}(\mathcal{K}(X),H)
	\leq
	\overline{{\rm dim}}(X,d).$$
\end{prop}
\begin{proof}
	Let $E\subset X$ be an $\varepsilon$-dense set of $X$. Consider $\mathcal{B}:=2^E\setminus\{\emptyset \}$, i.e., the collection of non-empty subset of $E$. It is easy to check that $\mathcal{B}\subset\mathcal{K}(X)$ is an $\varepsilon$-dense set. Hence, $N(\mathcal{K}(X),\varepsilon)\leq 2^{N(X,\, \varepsilon)}$, which implies $\overline{\text{mo}}(\mathcal{K}(X),H)
	\leq
	\overline{\text{dim}}(X,d)$.
	
	For $\varepsilon>0$, set $N=8\lfloor \frac{S(X,\, \varepsilon)}{8} \rfloor$. Then $S(X,\varepsilon)-7\leq N\leq S(X,\varepsilon)$. Pick a subset $E=\{ x_1,\ldots,x_N \}\subset X$ that is $\varepsilon$-separated. Let $F, F'$ be the sets in the proof of Lemma \ref{ky}. For each $\phi\in F'$, set $$B_\phi=\{ x_i:\phi(i)=1 \},$$
	and $\mathcal{F}:=\{ B_\phi:\phi\in F' \}.$ We see at once that $\mathcal{F}\subset\mathcal{K}(X)$ is $\varepsilon$-separated. Hence we have $S(\mathcal{K}(X),\varepsilon)\geq e^{cN}\geq e^{c(S(X,\varepsilon)-7)}$, which leads to $\underline{\text{dim}}(X,d)
	\leq
	\underline{\text{mo}}(\mathcal{K}(X), H).$
\end{proof}

For $\varepsilon>0$, we say that $B, C\in\mathcal{K}(X)$ are $\varepsilon$-split if $$\min\{d(x,y) :x\in B, y\in C\}>\varepsilon,$$ i.e., $d(B,C)>\varepsilon$. Denote by $B(\mathcal{A},\varepsilon)$ the maximal number of pairwise $\varepsilon$-split elements in $\mathcal{A}\subset\mathcal{K}(X)$. A similar proof of Lemma \ref{ky} shows that:
\begin{lem}\label{key2}
	Let $\mathcal{A}$ be a subset of $\mathcal{K}(X)$ satisfying the finite union property. Then $$\underline{{\rm mo}}(\mathcal{A},H)
	\geq
	\liminf_{\varepsilon\to0}\frac{\log B(\mathcal{A},\varepsilon)}{-\log\varepsilon}.$$
\end{lem}

Let $M$ be a Riemannian manifold, $U\subset M$ be an open set and $f:U\to M$ be $C^{1+\alpha}$ map which leaves invariant a compact subset $K$. Call $(K,f)$ is a \emph{conformal expanding repeller} if $f$ is conformal and expanding at $K$ (that is, for any $x\in K$, the derivative $Df(x)$ expands the Riemannian metric by a scalar factor greater than $1$). Then the box dimension of $\Lambda$ exists and coincides with its Hausdorff dimension (see \cite[Cro. 9.1.7]{Przytycki2010}).
\begin{proof}[Proof of Theorem \ref{conformal}]
	From definition of metric order, we have
	$$\limsup_{\varepsilon\to0}\frac{\log\log N^{f|\Lambda}(\varepsilon)}{-\log\varepsilon}=
	\overline{\text{mo}}(\mathcal{K}_{f|\Lambda }(\Lambda),H).$$
	By Proposition \ref{dimmo},
	$$\overline{\text{mo}}(\mathcal{K}_{f|\Lambda }(\Lambda),H)
	\leq
	\overline{\text{mo}}(\Lambda,H)
	\leq
	\overline{\text{dim}}(\Lambda)=b.$$

	On the other hand, there exists an ergodic invariant measure $\mu$ support on $K$ of maximal dimension, such that
	$b=\frac{h_\mu}{\chi_\mu}$, where $\chi_\mu$, $h_\mu$ denote the lyapunov exponent, the measure entropy of $\mu$ respectively (\cite[$\mathcal{x}$ 9.1]{Przytycki2010}). 
	For $\varepsilon_n=e^{(\chi_\mu+3\delta)(n+1)}$, $n\in\mathbb{N}$, there exists a set $G_n$ composed by $(n+n_0)$-period points, and $\#G_n\geq e^{(h_\mu-\delta)n}$; set $\Pi_n=\cup_{k\geq0}f^k(G_n)$, then $\Pi_n$ is $(d,\varepsilon_n)$-separated (\cite[$\mathcal{x}$ 2.1]{Berger2019}).
	
	Let $\mathcal{B}_n=\{ \text{orb}_f(y):y\in G_n \}\subset\mathcal{K}_{f|\Lambda}(\Lambda)$, where $\text{orb}_f(y)=\{f^ky:k\in\mathbb{N}_0 \}$. Then $\mathcal{B}_n$ is $\varepsilon$-split. Moreover,
	$$\#\mathcal{B}_n\geq \frac{\#G_n}{n+n_0}\geq e^{(h_\mu-2\delta)n},$$ for large $n$.
	
	Now, given $\varepsilon>0$ sufficiently small, take $n$ such that $\varepsilon_n\leq\varepsilon\leq\varepsilon_{n-1}$. We have $B(\mathcal{K}_{f|\Lambda }(\Lambda),\varepsilon)\geq B(\mathcal{K}_{f|\Lambda }(\Lambda),\varepsilon_n)\geq\#\mathcal{B}_n$.
	Thus
	$$\frac{\log B(\mathcal{K}_{f|\Lambda }(\Lambda),\varepsilon)}{-\log\varepsilon}
	\geq
	\frac{\#\mathcal{B}_n}{-\log\varepsilon}\geq
	\frac{h_\mu-2\delta}{\chi_\mu+3\delta}.$$
	So Lemma \ref{key2} yields 
	$\underline{\text{mo}}(\mathcal{K}_{f|\Lambda }(\Lambda),H)
	\geq
	(h_\mu-2\delta)/(\chi_\mu+3\delta)$. As $\delta$ is arbitrarily chose to $0$, we conclude that 
	$\underline{\text{mo}}(\mathcal{K}_{f|\Lambda }(\Lambda),H)
	\geq{h_\mu}/{\chi_\mu}=b$.
\end{proof}

\section{Further argument}
	Recently, we notice a paper {\rm\cite{Kiriki2019}} on arXiv, which defined the pointwise emergence of $f$ at point $x\in X$ by
	\begin{align*}
	\mathcal{E}_x(\varepsilon)&=\min\left\{ N: \exists\mathcal{F}\subset\mathcal{M}(X)\ \text{with}\ \#\mathcal{F}=N\ \text{such that}\ \limsup_{n\to\infty} W_1(e^f_n(x),\mathcal{F})\leq\varepsilon \right\}\\
	&=N(V(x),\varepsilon),
	\end{align*}
	and proved the following result:
	\begin{thm}{\rm\cite{Kiriki2019}}
		Let $(X,f)$ be full shift on $m$ symbols. Then there exists a residual subset $R\subset X$ such that for all $x\in R$
		\begin{equation*}
		\lim_{\varepsilon\to0}\frac{\log\log \mathcal{E}_x(\varepsilon)}{-\log\varepsilon}={\rm dim}(X).
		\end{equation*}
	\end{thm}
    If we analogously define the dynamical pointwise emergence by 
    \begin{align*}
    \mathcal{E}_x(n,\varepsilon)&=\min\left\{ N: \exists\mathcal{F}\subset\mathcal{M}(X)\ \text{with}\ \#\mathcal{F}=N\ \text{such that}\ \limsup_{n\to\infty} W_1^n(V(x),\mathcal{F})\leq\varepsilon \right\}\\
    &=N_\mathcal{M}(V(x),n,\varepsilon).
    \end{align*}
    By ergodic decomposition theorem, there exists a Borel set $X_0\subset X$ such that for any $\mu\in\mathcal{M}_f(X)$, $\mu(X_0)=1$ and for each $x\in X_0$, $V(x)\subset\mathcal{M}_f^{erg}(X)$ is a single point set. Thus $\mathcal{E}_x(n,\varepsilon)=1$ for any $x\in X_0$.

    If $(X,f)$ has Bowen specification property, then the set $\{ x\in X: V(x)=\mathcal{M}_f(X) \}$ is residual in $X$ (\cite{Sigmund1974}). As a result of Theorem \ref{expansive}, we have
    \begin{cro}
    	If $(X,f)$ is positive expansive and has specification property, then there exists a residual set $R\subset X$ such that for all $x\in R$,
    	\begin{equation*}
    		\lim_{\varepsilon\to0}\limsup_{n\to\infty}\frac{\log\log\mathcal{E}_x(n,\varepsilon)}{n}=h_{top}(f).
    	\end{equation*}
    \end{cro}

\noindent\textbf{Acknowledgments:} The authors are grateful to Prof. Hanfeng Li for pointing out an error in an early version of this paper.
The work was supported by the NNSF of China (11971236, 11601235, 11671208, 11431012), the NSF of Jiangsu Province (BK20161014), the China Postdoctoral Science Foundation (2016M591873) and the China Postdoctoral Science Special Foundation (2017T100384).  The work was also funded by the Priority Academic Program Development of Jiangsu Higher Education Institutions.

\end{document}